\documentclass[12pt]{article}
\usepackage{amsmath}
\usepackage{amsthm}
\usepackage{amssymb}
\usepackage{graphicx}
\theoremstyle{plain}

\newtheorem{theorem}{Theorem}

\newtheorem{lemma}{Lemma}

\usepackage{amssymb}
\usepackage{latexsym}
\usepackage{amsmath}
\usepackage{verbatim}
\usepackage{amsthm}
\usepackage{parskip}
\usepackage{epstopdf}

\newtheoremstyle{slplain}
  {.5\baselineskip\@plus.2\baselineskip\@minus.2\baselineskip}
  {.5\baselineskip\@plus.2\baselineskip\@minus.2\baselineskip}
  {\slshape}
  {}
  {\bfseries}
  {.}
  { }
  {}

\newcommand{\setB}{\mathcal{B}(n,r)}
\newcommand{\setS}{\mathfrak{S}}
\DeclareMathOperator{\maj}{maj}
\DeclareMathOperator{\des}{des}

\begin{document}

\title{On super Catalan polynomials}    

\author{Emily\ Allen \thanks{Department of Mathematical Sciences, Carnegie Mellon University, Pittsburgh, PA 15213, {\tt <eaallen@andrew.cmu.edu>}}\\
 Irina\ Gheorghiciuc \thanks{Department of Mathematical Sciences, Carnegie Mellon University, Pittsburgh, PA 15213, {\tt <gheorghi@andrew.cmu.edu>}}}
\date{}

\maketitle

\begin{abstract}
We present a $q$-analog of the super Catalan number $(2m)!(2n)!/2m!n!(m+n)!$, which also generalizes the $q$-Catalan numbers $c_n(\lambda)$, due to F\"urlinger and Hofbauer, for $\lambda=0$ and $\lambda=1$. We give a combinatorial interpretation for this analog when $m=2$.
\end{abstract}

\section{Introduction}

In \cite{SuperBallot} Gessel reintroduces the integers 
\[ S(m,n) = \frac{{2m \choose m}{2n \choose n}}{{m + n \choose n}} = \frac{(2m)!(2n)!}{m!n!(m+n)!},\]
which were studied by Eugene Catalan in 1874 \cite{cat}. When $m,n>0$, the numbers $S(m,n)$ are even. Gessel and Xin refer to $T(m,n) = S(m,n)/2$ as the super Catalan numbers \cite{GesselXin}. They gave a combinatorial interpretation of $T(2,n)$ in terms of pairs of Dyck paths with restricted heights. The set of pairs of Dyck paths considered by Gessel and Xin in \cite{GesselXin} is related to the Dyck paths in Theorem. 

We will use the standard $q$-notation
\[[r]_q = 1 + q + \cdots + q^{r-1} \qquad \text{ and } \qquad [n]!_q = \prod_{r=1}^n [r]_q.\]
The polynomials
\[S_q(m,n) = \frac{[2m]!_q[2n]!_q}{[m]!_q[n]!_q[m+n]!_q}\] 
have been studied by Warnaar and Zudilin \cite{warnaar}, and by Allen \cite{Emily}. Warnaar and Zudilin proved that $S_q(m,n)$ are polynomials with nonnegative integer coefficients \cite{warnaar}. Allen conjectured that $S_q(m,n)$ are unimodal \cite{Emily}.

Let $\mathfrak{S}(m,n)$ denote the set of lattice paths that begin at the origin, have $m$ \textit{up} steps (drawn by ascending edges) and $n$ \textit {down} steps (drawn by descending edges). Let $\mathfrak{S}_+(m,n)$ denote the subset of those paths which never go below the $x$-axis. We will refer to $\mathfrak{S}_+(n,n) = \mathcal{C}_n$ as the set of Catalan paths of length $2n$. The height of $\pi \in \mathfrak{S}(m,n)$, which is the maximum level $\pi$ reaches, will be denoted by $h(\pi)$. A path of length $\ell$ can be represented as a sequence of zeros and ones, $\pi = \pi_1\cdots \pi_{\ell}$, where zeros represent \textit{up} steps and ones represent \textit{down} steps.

For a lattice path $\pi$ we define the descent set $D(\pi)$, the major index $\maj(\pi)$, and the descent index $\des(\pi)$ to be
\[D(\pi) = \{i : \pi_i > \pi_{i+1}, 1 \leq i \leq \ell - 1\},\]
\[\maj(\pi) = \sum_{i \in D(\pi)} i ,\]
\[\des(\pi) = |D(\pi)|.\]

A combinatorial interpretation of $S_q(0,n)$ was given by MacMahon \cite{MacMahon}
\[S_q(0,n) = \Big[\big.^{2n}_n\big.\Big]_q = \sum_{\pi \in \mathfrak{S}(n,n)} q^{\maj(\pi)}.\]
Allen showed that $T_q(m,n) = S_q(m,n)/(1+q^n)$ is a polynomial \cite{Emily}. In fact \[T_q(1,n) = \sum_{\pi \in \mathfrak{S}_+(n,n)} q^{\maj(\pi) - \des(\pi)} \,\,\,\,\,\, \mbox{and} \,\,\,\,\,\,  T_q(n,1)  = \sum_{\pi \in \mathfrak{S}_+(n,n)} q^{\maj(\pi)} \] are respectively $q$-Catalan numbers $c_n(0)$ and $c_n(1)$ defined by F\"urlinger and Hofbauer \cite{fh}.

The super Catalan numbers satisfy the following identity, attributed to Dan Rubenstein \cite{SuperBallot}
\begin{equation}\label{rec}
 4T(m,n) = T(m+1,n) + T(m,n+1).
\end{equation}

The following $q$-analog of this identity holds
\begin{equation}(1+q^n)(1+q^{n-m})T_q(m,n) = q^{n-m}T_q(n,m+1) + T_q(m,n+1).
\end{equation}

In Section 2 we provide several results on $q$-Ballot Numbers. In Section 3 we expand on our methods in \cite{art1} to give a combinatorial interpretation of $T_q(2,n)$.

\section{A $q$-Ballot Number}

Let $\setB$ denote the set of paths of length $2n$ which begin at the origin with an \textit{up} step, end at $(2n,-2r+2)$, and never go below the line $y=-2r+2$. In particular $\mathcal{B}(n,1) = \mathcal{C}_n$, the set of Catalan paths of length $2n$. 

Define the $q$-Ballot Number
\[B_q(n,r) = \frac{[2n-1]!_q[2r]_q}{[n+r]_q![n-r]_q!} = \frac{1}{q^{n-r}}\left( \Big[\big.^{2n-1}_{n+r-1}\big.\Big]_q - \Big[\big.^{2n-1}_{n+r}\big.\Big]_q\right). \]

A different $q$-Ballot Number has been defined by Chapoton and Zeng in \cite{chapoton}.

 \begin{lemma}\label{reflect} Let $\pi \in \setS(m,n)$ ending with an \textit{up} step. Reflecting $\pi$ over the $x$-axis gives a path $\rho \in \setS(n,m)$ ending with a \textit{down} step which satisfies $\maj(\pi) = \maj(\rho) + n$.
 \end{lemma}
 
 \proof{Given a path $\pi \in \setS(m,n)$ ending with an \textit{up} step, let $D(\pi) = \{X_1,\ldots, X_{\ell}\}$. We let $X_0=0$. Define $u_i$ and $d_i$ to be the number of \textit{up} and \textit{down} steps, respectively, between indices $X_i$ and $X_{i+1}$. Let $\rho$ be the reflection of $\pi$ across the $x$-axis. Then the descents of $\rho$ occur exactly at indices $X_i+u_i$ for $i < \ell$. Hence,
 
 $\maj(\rho) = \sum_{i=0}^{\ell - 1}( X_i + u_i) = \maj(\pi) + m - (X_{\ell} + u_{\ell}) = \maj(\pi) + m - (n+m) = \maj(\pi) - n$. 
}

\begin{theorem}\label{qballot}
 \begin{equation} B_q(n,r) = \sum_{\pi \in \setB} q^{\maj(\pi) - \des(\pi)}\end{equation}

\end{theorem}
 
 \begin{proof} 
 Let $\mathfrak{S}_{>}$ denote the set of paths in $\mathfrak{S}(n+r-1, n-r)$  which have height strictly greater than $2r-1$, and let $\mathfrak{S}_{\leq}$ denote the set of paths in $\mathfrak{S}(n+r-1,n-r)$ which never go above $y=2r-1$. 
 
 We will define a bijection $\psi: \mathfrak{S}_> \rightarrow \mathfrak{S}(n+r,n-r-1)$ which preserves the major index. Given a path $\pi \in \mathfrak{S}_{>}$, let $R$ be the right-most highest point on $\pi$. Since $\pi$ has height strictly greater than $2r-1$ and ends at level $2r-1$, the point $R$ is not the last point on $\pi$. Let $RL$ be the \textit{down} step following $R$. Define $\psi(\pi)$ to be the path obtained from $\pi$ by changing the \textit{down} step $RL$ into an \textit{up} step. Note that $\psi(\pi) \in \mathfrak{S}(n+r,n-r-1)$  and $L$ is the left-most highest point on $\psi(\pi)$.
 To see that $\psi$ is a bijection from $\mathfrak{S}_{>}$ to $\mathfrak{S}(n+r,n-r-1)$, given a path $\rho$ in $\mathfrak{S}(n+r,n-r-1)$, locate the left-most highest point $L$ on $\rho$ and change the \textit{up} step preceding it into a \textit{down} step to obtain $\pi$. Since $\rho$ has height at least $2r+1$, the path  $\pi$ will have height at least $2r$. Therefore $\pi \in  \mathfrak{S}_>$ and $\psi(\pi) = \rho$. 
 The bijection $\psi$ preserves the major index because the descent sets of $\pi$ and $\psi(\pi)$ are the same.
  It follows that
 
 \[\Big[\big.^{2n-1}_{n+r-1}\big.\Big]_q - \Big[\big.^{2n-1}_{n+r}\big.\Big]_q = \sum_{\pi \in \mathfrak{S}(n+r-1,n-r)} q^{\maj(\pi)} - \sum_{\pi \in \mathfrak{S}(n+r,n-r-1)}q^{\maj(\pi)} = \sum_{\pi \in \mathfrak{S}_{\leq}} q^{\maj(\pi)}.\]

 We will define a bijection $\varphi: \mathfrak{S}_{\leq} \rightarrow \setB$. Let $\pi \in \mathfrak{S}_{\leq}$. Define $\varphi(\pi)$ to be the path obtained from $\pi$ by reflecting $\pi$ across the $x$-axis and
 then adding an \textit{up} step to the beginning of the path. This is clearly a bijection from $\mathfrak{S}_{\leq}$ to $\setB$. By Lemma \ref{reflect}, reflecting $\pi$ across the $x$-axis causes the major index to decrease by $n-r$. Adding an \textit{up} step to the beginning of the path increases all descents by $1$, hence the major index of the reflection of $\pi$ equals the major minus descent index of $\varphi(\pi)$. It follows that, 
 
\[\frac{1}{q^{n-r}}\left( \Big[\big.^{2n-1}_{n+r-1}\big.\Big]_q - \Big[\big.^{2n-1}_{n+r}\big.\Big]_q \right)   = \sum_{\pi \in \mathfrak{S}_{\leq}} q^{\maj(\pi) - (n-r)} = \sum_{\pi \in \setB} q^{\maj(\pi) - \des(\pi)}.\]\end{proof}

The following is a generalization of Eq. 2 in \cite{art1}. The proof is algebraic and can be found in \cite{thesis}.

\begin{equation}
\label{qBallotIdEq}
q^{(n-1)(m-1)} T_q(m,n) = \sum_{r=1}^m (-1)^{r-1}q^{{r-1 \choose 2}} \frac{1+q^m}{1+q^r}B_q(n,r)B_q(m,r)
\end{equation}

\section{Combinatorial Interpretation}

When m=2, Eq. \ref{qBallotIdEq} becomes

 \begin{equation} \label{id1} q^{n-1}T_q(2,n)= (1+q^2)B_q(n,1) - B_q(n,2).  \end{equation}

For a path $\pi \in \mathcal{C}_n$, let $X$ be the last, from left to right, level one point up to and including the right-most maximum $R$ on $\pi$. Let $h_{-}(\pi)$ denote the maximum level that the path $\pi$ reaches from its beginning until and including point $X$, and $h_{+}(\pi)$ denote the maximum level that the path $\pi$ reaches after and including point $X$. Obviously $h_-(\pi) \leq h_+(\pi) = h(\pi)$.  Let $\Omega_n$ denote the set of $\pi \in \mathcal{C}_n$ such that $h_{+}(\pi) \leq h_{-}(\pi) + 2$.

\begin{theorem}
 \begin{equation}
  T_q(2,n) = q^{n-1} +  q^{3-n}\sum_{\pi \in \Omega_n} q^{\maj(\pi) - \des(\pi)}. 
 \end{equation}

\end{theorem}\begin{proof}
By $\mathcal{B}^*(n,2)$ we denote the set of paths in $\mathcal{B}(n,2)$ which do not attain level $y=-1$ before their right-most maximum. Let $\mathcal{B}^{**}(n,2)=\mathcal{B}(n,2)-\mathcal{B}^*(n,2)$ and 

$$B^*_q(n,2)=\sum_{\pi \in \mathcal{B}^*(n,2)} q^{\maj(\pi) - \des(\pi)}; \,\,\,\,\,\,\,\,\ B^{**}_q(n,2)=\sum_{\pi \in \mathcal{B}^{**}(n,2)} q^{\maj(\pi) - \des(\pi)}.$$

By Theorem \ref{qballot} and Eq. \ref{id1}
$$q^{n-1}T_q(2,n) = (B_q(n,1) - B^*_q(n,2))+(q^2B_q(n,1)- B^{**}_q(n,2)).$$

First we compute $B_q(n,1) - B^*_q(n,2).$ For $\pi \in \mathcal{B}^*(n,2)$, let $RQ$ be the {\it down} step that follows the right-most maximum point $R$ of $\pi$. We define $f(\pi)$ to be the path obtained by substituting the {\it down} step $RQ$ by an {\it up} step. See Figure~\ref{Mn+1new}. Note that $f(\pi) \in \mathcal{C}_n$ and, since at least two {\it up} steps precede $Q$ on $f(\pi)$, the height of $f(\pi)$ is at least two. Also $\pi$ and $f(\pi)$ have the same set of descents, thus $\des(\pi)=\des(f(\pi))$ and $\maj(\pi)=\maj(f(\pi))$. It is important to mention that $Q$ is the left-most maximum on $f(\pi)$.

\begin{figure}[!ht]
\begin{center}
\includegraphics[scale=0.5]{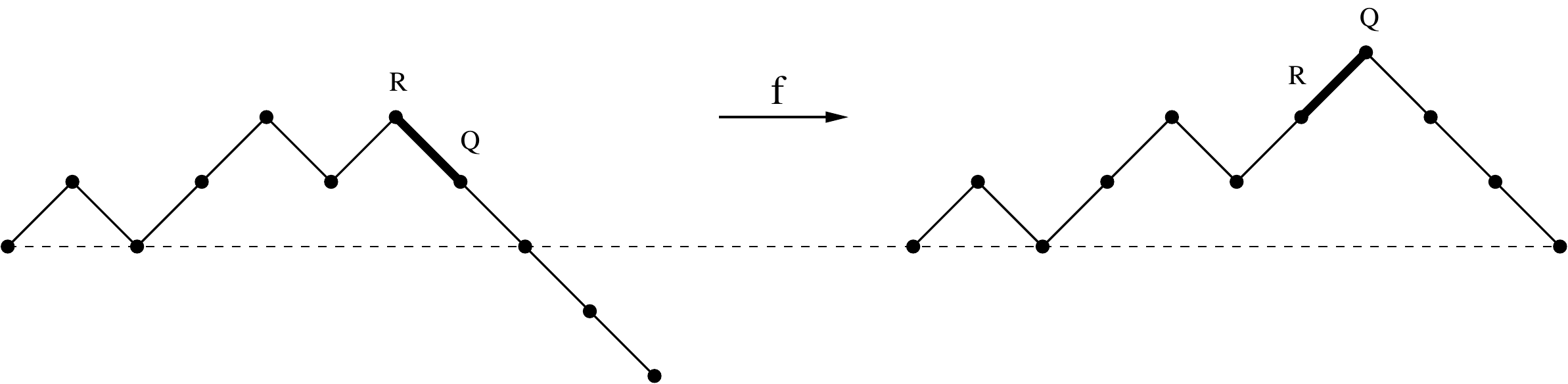}
\caption{$f$ substitutes the {\it down} step $RQ$ by an {\it up} step}
\label{Mn+1new}
\end{center}
\end{figure}

We will show that $f$ is a bijection between $\mathcal{B}^*(n,2)$ and the set of paths $\rho$ in $\mathcal{C}_n$ of height $h(\rho)>1$. Let $Q$ be the left-most maximum on $\rho$ and $RQ$ be the {\it up} step that precedes $Q$. Substitute the {\it up} step $RQ$ by a {\it down} step, which makes $R$ the right-most maximum of the resulting path, call it $\pi$. Note that $\pi \in \mathcal{B}^*(n,2)$ and $f(\pi)=\rho$.

It follows that $$B_q(n,1) - B^*_q(n,2)=\sum_{\substack{\pi \in \mathcal{C}_n \\ h(\pi)=1}} q^{\maj(\pi) - \des(\pi)}=q^{(n-1)^2}.$$

We define a {\it descent point} to be a point on a path which is preceded by a \textit{down} step, and followed by a \textit{up} step.
We define a {\it down wedge sequence} to be a portion of a path that starts with a {\it down} step, alternates between {\it down} steps and {\it up} steps, and ends with an {\it up} step. See Figure~\ref{Seq}.

\begin{figure}[!ht]
\begin{center}
\includegraphics[scale=0.6]{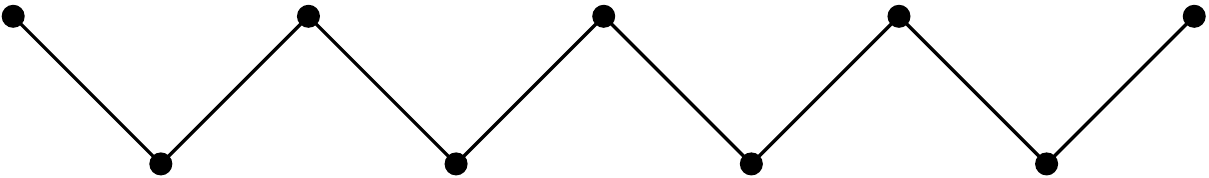}
\caption{A down wedge sequence}
\label{Seq}
\end{center}
\end{figure}

We find a combinatorial interpretation for $q^2B_q(n,1) - B^{**}_q(n,2)$ by establishing an injection $g$ from $\mathcal{B}^{**}(n,2)$ to $\mathcal{C}_n$. A path $\pi$ in $\mathcal{B}^{**}(n,2)$ attains $y=-1$ before its right-most maximum point $R$. Let $N$ be the first point before $R$ at which $\pi$ attains $y=-1$. We consider two cases: when $N$ is immediately followed by a {\it down} step, and when $N$ is immediately followed by an {\it up} step.

{\bf Case one: $N$ is immediately followed by a {\it down} step.} Since $N$ is the left-most point on $\pi$ on level $y=-1$, $N$ is preceded by two {\it down} steps and is followed by one {\it down} step, then one {\it up} step. See Figure~\ref{M1new}. Let $MN$ be the {\it down} step that precedes $N$ and $NY$ be the {\it down} step that follows $N$. Substitute $MY$ by two {\it up} steps. The resulting path is a ballot path of length $2n$ that ends at level two. Rename $N$ to be $X$. From left to right, $X$ is the last level one point on this ballot path. The maximum level that this ballot path reaches up to and including point $X$ is less than the maximum level it reaches after and including point $X$ by at least 4.

Let $L$ be the left-most maximum point of this ballot path and $QL$ be the {\it up} step that precedes $L$. Substitute the {\it up} step $QL$ by a {\it down} step. See Figure ~\ref{M1new}. The resulting path $g(\pi)$ is in $\mathcal{C}_n$ and $Q$ is its right-most maximum. Point $X$ is the last level one point on $g(\pi)$ before its right-most maximum $Q$ and the point on $g(\pi)$ before $X$ is a descent. Note that $h_+(g(\pi)) \geq h_-(g(\pi)) + 3$. Also $\des(\pi)=\des(g(\pi))$ and $\maj(\pi)=\maj(g(\pi))+2$. Thus $\maj(\pi)-\des(\pi)=\maj(g(\pi))-\des(g(\pi))+2.$

\begin{figure}[!ht]
\begin{center}
\includegraphics[scale=0.6]{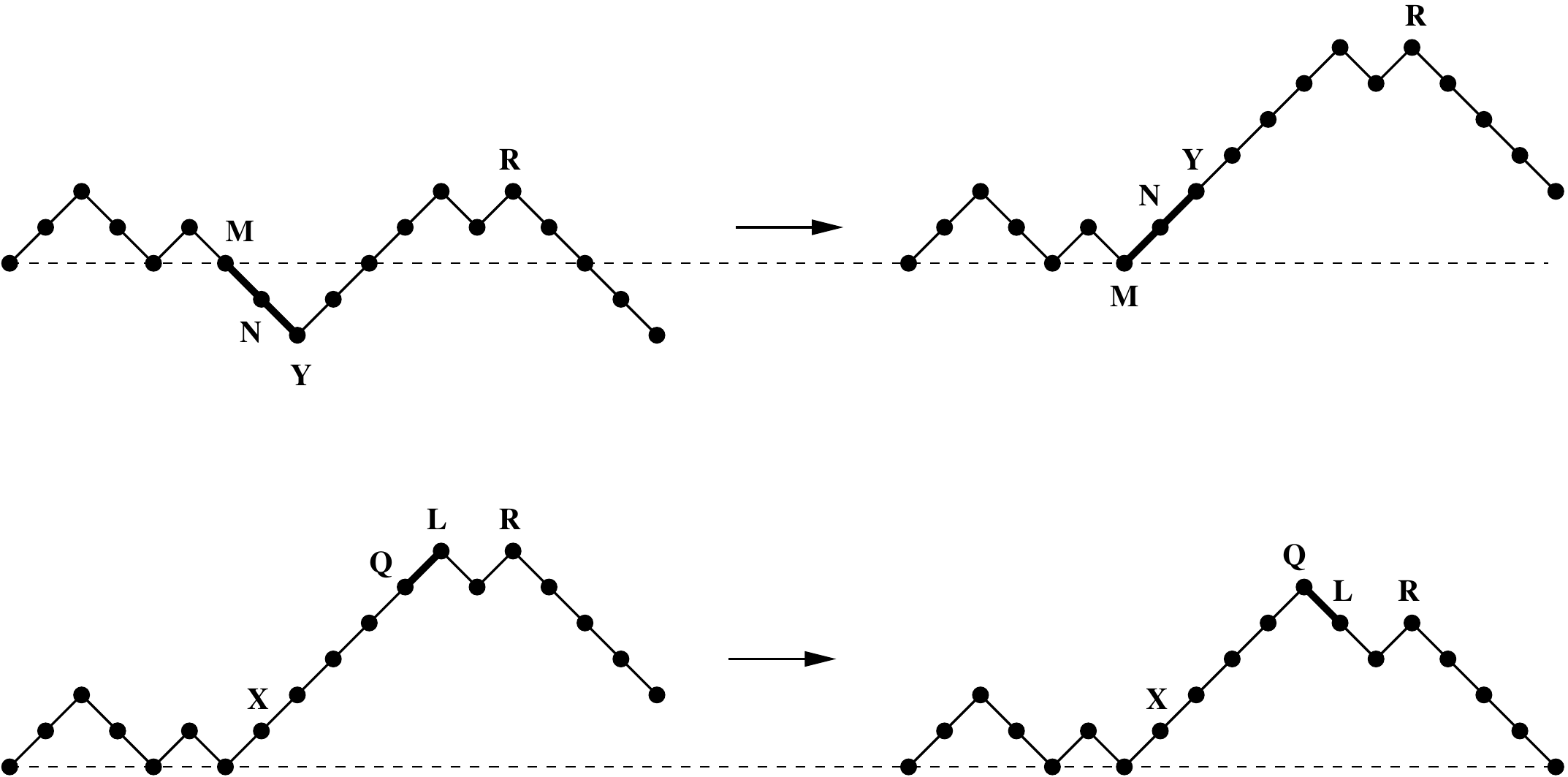}
\caption{The action of $g$ when $N$ is followed by a {\it down} step}
\label{M1new}
\end{center}
\end{figure}

{\bf Case two: $N$ is immediately followed by an {\it up} step.} Since $N$ is the left-most point on $\pi$ on level $y=-1$, $N$ is preceded by two {\it down} steps which form a segment we denote by $XN$. See Figure \ref{fix}. Let $\sigma$ be the longest, possibly empty, down wedge sequence that precedes $X$. Let $Y$ be the first point of the sequence $\sigma$. Note that either $Y$ is the second point on $\pi$ or $Y$ is preceded by a {\it down}  step. Remove $\sigma$ from its original position and insert it immediately after $N$. Then substitute $XN$ by two {\it up} steps. The resulting path is a ballot path of length $2n$ that ends at level two. From left to right, $X$ is the last level one point on this ballot path. The maximum level that this ballot path reaches up to and including point $X$ is less than the maximum level it reaches after and including point $X$ by at least 4.

Let $L$ be the left-most maximum point of this ballot path and $QL$ be the {\it up} step that precedes $L$. Substitute the {\it up} step $QL$ by a {\it down} step. See Figure ~\ref{fix}. The resulting path $g(\pi)$ is in $\mathcal{C}_n$ and $Q$ is its right-most maximum. Note that $X$ is the last level one point on $g(\pi)$ before its right-most maximum $Q$ and the point on $g(\pi)$ before $X$ is NOT a descent.  Also $h_+(g(\pi)) \geq h_-(g(\pi)) + 3$. If $Y$ is the second point on the original path $\pi$, then $g$ removes the original descent of $\pi$ that occurs immediately after $Y$ and moves the descent that originally corresponds to $N$ one unit to the left. Thus $\des(\pi)=\des(g(\pi))+1$ and $\maj(\pi)=\maj(g(\pi))+3$. If $Y$ is NOT the second point on the original path $\pi$ and $\sigma$ is not empty, then $g$ moves the descent that originally occurs immediately after $Y$ and the one that corresponds to $N$ one unit to the left. If $Y$ is NOT the second point on the original path $\pi$ and $\sigma$ is empty, then $g$ moves the descent that corresponds to $N$ two units to the left. Thus $\des(\pi)=\des(g(\pi))$ and $\maj(\pi)=\maj(g(\pi))+2$. In all the cases $\maj(\pi)-\des(\pi)=\maj(g(\pi))-\des(g(\pi))+2.$ 

\begin{figure}[!ht]
\begin{center}
\includegraphics[scale=0.6]{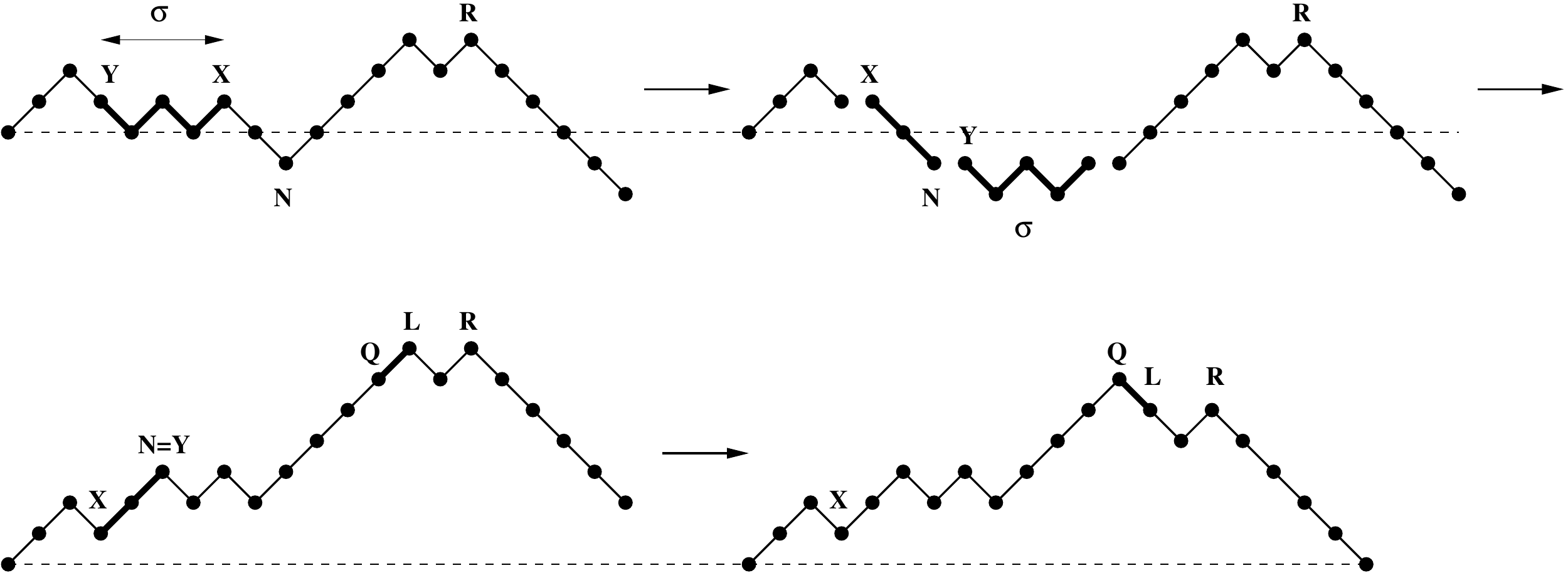}
\caption{The action of $g$ when $N$ is followed by an {\it up} step}
\label{fix}
\end{center}
\end{figure}

If a path $\rho$ is in the image of $g$, then $h_{+}(\rho) \geq h_{-}(\rho)+3$, thus $\rho \in \mathcal{C}_n-\Omega_n$. We will show that the image of $g$ is $\mathcal{C}_n-\Omega_n$. Let $\rho$ be in $\mathcal{C}_n$ and $h_{+}(\rho) \geq h_{-}(\rho)+3$. Let $Q$ be the right-most maximum on $\rho$ and $QL$ be the {\it down} step that follows $Q$. Substitute the {\it down} step $QL$ by an {\it up} step. The result is a ballot path of length $2n$ that ends at level two. Note that $L$ is the left-most maximum on this ballot path. Let $R$ denote the right-most maximum on this ballot path. From left to right, let $X$ be the last level one point on this ballot path. The maximum level that this ballot path reaches up to and including point $X$ is less than the maximum level it reaches after and including point $X$ by at least 4. We consider two cases: when the point before $X$ is a descent, and when the point before $X$ is not a descent.

If the point before $X$ is a descent, let $M$ be that descent. Let $Y$ be the point that follows $X$. Since $X$ is the last point on level one before $R$, $XY$ is an {\it up} step. Substitute $MY$ with two {\it down} steps. We will call the resulting path $\pi$. Note that $\pi$ attains level $y=-1$ before its right-most maximum $R$ and, immediately after attaining level $y=-1$ for the first time, it attains level $y=-2$. Thus $\pi$ is in $\mathcal{B}^{**}(n,2)$, falls into Case 1, and $g(\pi)=\rho$.

Next we consider the case when the point before $X$ is NOT a descent. Note that, since $X$ is the last level one point before $R$, $X$ is followed by two {\it up} steps. Let $XY$ be the segment that consists of these two {\it up} steps. Let $\sigma$ be the longest, possibly empty, down wedge sequence that starts at $Y$. Note that $\sigma$ is followed by an {\it up} step. Remove $\sigma$ from its original position and insert it immediately before $X$, then substitute $XY$ with two {\it down} steps. We will call the resulting path $\pi$. Note that $\pi$ attains level $y=-1$ for the first time at $Y$, before its right-most maximum $R$, and $Y$ is followed by an {\it up} step. Thus $\pi$ is in $\mathcal{B}^{**}(n,2)$, falls into Case 2, and $g(\pi)=\rho$.

It follows that $$q^2B_q(n,1) - B^{**}_q(n,2)=q^2 \, \sum_{\pi \in \Omega_n} q^{\maj(\pi) - \des(\pi)}. $$
\end{proof}

\section*{Acknowledgements}

We are very thankful to Ira Gessel and Bogdan Ion for their helpful comments. The problem about the combinatorial interpretation of the super Catalan Numbers was first mentioned to us by the late Herb Wilf, to whom we are immensely grateful.

%
%
%

\end{document}